\documentclass{amsart}
\usepackage[utf8]{inputenc}
\usepackage[english]{babel}
\usepackage{csquotes}
\usepackage{graphicx}
\usepackage{indentfirst}
\usepackage{overpic}
\usepackage[dvipsnames]{xcolor}
\usepackage[export]{adjustbox}
\usepackage{amsthm}
\usepackage{amsfonts}
\usepackage{fullpage}
\usepackage{booktabs}
\usepackage{mathtools}
\usepackage[shortlabels]{enumitem}
\usepackage{float}
\usepackage{url}
\usepackage{moreenum}
\usepackage{multicol}
\usepackage{xtab}
\usepackage{array}
\usepackage{multirow}
\usepackage{url}
\usepackage{makecell}
\usepackage[margin=1in]{geometry}
\newcommand{\gds}{$g_{ds}(K)$}
\newcommand{\s}[1]{$S^#1$}
\newcommand{\kn}{$K$}
\newcommand{\ita}[1]{\textit{#1}}
\newcommand{\g}[1]{$g_#1(K)$}
\newcommand{\gt}[2]{$#1g_#2(K)$}

\newtheorem{theorem}{Theorem}[section]
\newtheorem{lemma}[theorem]{Lemma}
\newtheorem{corollary}[theorem]{Corollary}
\newtheorem*{L1}{Lemma~\ref{lemma:p_n}}
\newtheorem*{L2}{Corollary~\ref{corollary:unknotting number}}
\newtheorem*{L3}{Lemma~\ref{lemma:genus}}

\theoremstyle{definition}
\newtheorem{define}[theorem]{Definition}

\theoremstyle{remark}

\newtheorem{remark}[theorem]{Remark}

\title{\textbf{\MakeUppercase{\Large{Determining the doubly slice genera of prime knots with up to 12 crossings}}}}

\author{Lucia P. Karageorghis}
\email{lucia.karag@gmail.com}
\address{Department of Mathematical Sciences, Durham University, UK}

\author{Frank Swenton}
\email{fswenton@middlebury.edu}
\address{Department of Mathematics, Middlebury College, USA}

\begin{document}

\maketitle

\begin{abstract}
    For a knot \kn{}, the doubly slice genus \gds{} is the minimal $g$ such that \kn{} divides a closed, orientable, and unknotted surface of genus $g$ embedded in \s4. In this paper, we identify the doubly slice genera of 2909 of the 2977 prime knots which have a crossing number of 12 or fewer.
\end{abstract}

\section{Introduction}
A knot $K$ is a smooth, 1-dimensional submanifold of \s3 that is homeomorphic to \s1. If a closed, orientable surface $\Sigma$, which is smoothly embedded in \s4, has a cross-section \s3 $\cap$ $\Sigma=$ \kn{}, then we say that \kn{} \ita{divides} the surface. In addition, $\Sigma$ is \textit{unknotted} if it bounds a smoothly embedded, 3-dimensional handlebody in \s4. If we can find a smooth embedding of \s2 in \s4 which is divided by \kn{} and unknotted, then we say that \kn{} is \ita{doubly slice}. In their 2015 paper \cite{livingston2015doubly}, Livingston and Meier concluded their discussion of doubly slice knots by introducing the \ita{doubly slice genus}. Just as we consider knots that divide unknotted spheres, we can consider knots that divide unknotted surfaces of higher genera. Accordingly, we define the doubly slice genus, \gds{}, as the smallest possible $g$ such that \kn{} divides an unknotted, closed, and orientable surface of genus $g$ smoothly embedded in \s4. In this paper, we determine the doubly slice genera of 2909 of the 2977 prime knots that have a crossing number of 12 or fewer. We do this by utilising previously established bounds on the doubly slice genus, in particular those given by McDonald \cite{McDonald2019band}, which employ the band unknotting number and band number of a knot, and by using a computer to search for sequences of oriented saddle moves.

The \ita{3-genus} of a knot $K$, denoted \g3, is the smallest genus amongst all Seifert surfaces for \kn{}. Furthermore, the \ita{smooth 4-genus}, \g4, is the smallest $h$ such that there exists a closed and orientable surface smoothly embedded in $D^4$ which has boundary \kn{} and genus $h$. As observed in papers \cite{livingston2015doubly} and \cite{McDonald2019band}, we have the following lemma. \begin{lemma} \label{lemma:genus} For a knot \kn{}, $2g_4(K) \leq g_{ds}(K)\leq 2g_3(K).$
 \end{lemma} We include the proof of this lemma in Section \ref{subsection:genus}.

An \textit{oriented saddle move} is a move performed on a knot diagram as illustrated in Figure \ref{fig:OSM}. \begin{figure}[h]
    \centering
    \includegraphics{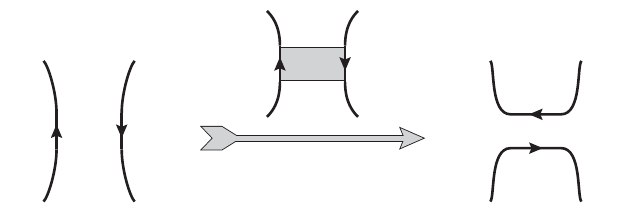}
    \caption{An oriented saddle move.}
    \label{fig:OSM}
\end{figure}
\begin{define} \label{def:unlinking_band_sequence} For a knot \kn{}, an \ita{unlinking band sequence} is a sequence of oriented saddle moves transforming \kn{} into an unlink. The \ita{band number}, $b(K)$, is the minimum length of all unlinking band sequences for $K$.\end{define}

In \cite{McDonald2019band}, McDonald gave an upper bound on the doubly slice genus as follows. \begin{theorem} \label{theorem:McDonald} For a knot \kn{}, $g_{ds}(K) \leq b(K).$ \end{theorem}
\begin{corollary} \label{corollary:bound}
If a knot $K$ has an unlinking band sequence of length $n$, then $g_{ds}(K)\leq n$.
\end{corollary}

The bounds on \gds{} from Lemma \ref{lemma:genus} and Corollary~\ref{corollary:bound} have proven to be efficacious for determining the doubly slice genera of prime knots. Section \ref{subsection:table} contains a table that presents the doubly slice genera which we determined, and Section \ref{section:results} details the methods used for these computations. In the absence of an exact value, we provide an interval within which the doubly slice genus lies.

A \ita{slice knot} is a knot $K$ such that $g_4(K)=0$, i.e., $K\subset S^3=\partial D^4$ bounds a disc that is smoothly embedded in $D^4$. It is evident that our \gt24 lower bound from Lemma \ref{lemma:genus} is not useful for such knots. Note that if a knot \kn{} is doubly slice, then \kn{} is slice. 
\begin{remark} \label{remark:slice_categories} We shall call a knot \textit{undetermined} if it is a slice knot for which our work does not resolve whether or not the knot is doubly slice. (Fortunately, only three of the knots that we are studying are undetermined thanks to \cite{livingston2015doubly,Piccirillo_2020}.) We thus form three classes of slice knots: doubly slice (\gds{} $=0$), slice but not doubly slice (\gds{} $\geq 1$), and undetermined.\end{remark}

By Lemma \ref{lemma:genus} and Corollary \ref{corollary:bound}, we can deduce the next corollary.
\begin{corollary} \label{corollary:sequence}
Given a slice but not doubly slice knot $K_1$, if $K_1$ has an unlinking band sequence with length 1, then $g_{ds}(K_1)=1$. Given a non-slice knot $K_2$, if $K_2$ has an unlinking band sequence with length $2g_4(K_2)$, then $g_{ds}(K_2)=2g_4(K_2)$.  \end{corollary}
Thus, for slice but not doubly slice and non-slice knots, our objective is to find unlinking band sequences with length as close as possible to 1 and to $2g_4(K)$, respectively. For undetermined knots, we are simply looking for unlinking band sequences of the shortest possible length. The automated search for unlinking band sequences is described in Section \ref{section:automate}.
\begin{remark} \label{remark:interval}
When we started this work, there were 21 knots for which the smooth 4-genus was only known to lie within the interval $[1,2]$. In this case, we can use Lemma \ref{lemma:genus} to assert a lower bound of 2 on their doubly slice genera. What's more, in the course of searching for the doubly slice genera of these 21 knots, we proved that the 4-genus of 18 of them is exactly 1 (see Lemma \ref{lemma:interval}).
\end{remark}

An \ita{unknotting sequence} is a sequence of crossing changes performed on \kn{} that results in the unknot. As an alternative to systematically searching for unlinking band sequences with a computer, one can find unlinking band sequences via unknotting sequences. 
\begin{define} \normalfont Given a knot \kn{}, a \ita{$(p,n)$ unknotting sequence} for \kn{} is an unknotting sequence containing $p$ positive crossing changes and $n$ negative crossing changes. \end{define}
\begin{lemma} \label{lemma:p_n} Given a knot \kn{}, if there exists a $(p,n)$ unknotting sequence of \kn{}, then there exists an unlinking band sequence with length $2\cdot \max\{p,n\}$. In other words, \gds{} $\leq 2\cdot \max\{p,n\}$. \end{lemma}
We prove Lemma \ref{lemma:p_n} in Section \ref{subsection:balance}. 

In practice, we did not use Lemma \ref{lemma:p_n} to its full strength since, when conducting this research, there was not relevant data available. We instead used a weaker upper bound which exploits the unknotting number, $u(K)$, for which there was data.
\begin{corollary} \label{corollary:unknotting number}
Given a knot \kn{}, there exists an unlinking band sequence with length $2u(K)$. In other words, \gds{} $\leq 2u(K)$. \end{corollary}
Corollary \ref{corollary:unknotting number} is also proven in Section \ref{subsection:balance}.

\subsection{Acknowledgements}
\hfill{} \\ \indent{} Lucia would like to express her deep gratitude to Mark Powell for his invaluable support and guidance, which made this paper possible. A thanks is also due to Patrick Orson for his assistance with nomenclature, and Lemma \ref{lemma:p_n}. Furthermore, she is thankful to the London Mathematical Society for affording her an Undergraduate Research Bursary and to the Department of Mathematical Sciences at Durham University for supporting the work. Finally, she would like to show her appreciation to Frank Swenton for his willingness to assist with this project, as well as his generosity with his time and his enthusiasm.

\section{Results}
Our findings are presented in Table 1, which can be found in Section \ref{subsection:table}. The superscript of each \gds{} value in Table 1 corresponds to an item in the list below, which gives the method we used to arrive at said value. Throughout this list, a knot is described as non-slice if its exact smooth 4-genus was known when we began our work, unless stated otherwise, and \kn{} refers to any knot in the ``\kn{}" column.
\label{section:results}

\begin{enumerate}[label=(\greek*)]
    \item \label{method:double_slice} By Livingston and Meier \cite{livingston2015doubly}, \kn{} is doubly slice, i.e., \gds{} $=0$.
    
    \item \label{method:genera_equal} \kn{} is non-slice and $g_4(K)=g_3(K)$ \cite{knotinfo}. By Lemma \ref{lemma:genus}, \gds{} $=2g_4(K)=2g_3(K)$.
    
    \item \label{method:unknotting_number} \kn{} is non-slice and $u(K)=g_4(K)$ \cite{knotinfo}. By Corollary \ref{corollary:unknotting number}, there exists an unlinking band sequence of \kn{} with length $2g_4(K)$. So, by Corollary \ref{corollary:sequence}, \gds{} $=2g_4(K)$.
    
    \item \label{method:optimal_sequence} \kn{} is \{(i) slice but not doubly slice, (ii) non-slice\} \cite{livingston2015doubly,knotinfo}, and we found an unlinking band sequence of \kn{} with length \{(i) $1$, (ii) $2g_4(K)$\}. Applying Corollary \ref{corollary:sequence}, \gds{} $=$ \{(i) $1$, (ii) $2g_4(K)$\}.
    
    \item \label{optimal_sequence_Frank} \kn{} is slice but not doubly slice \cite{livingston2015doubly}, and by \cite{owens2021algorithm} there exists an unlinking band sequence of \kn{} with length $1$. By Corollary \ref{corollary:sequence}, \gds{} $=1$.
    
    \item \label{method:optimal_sequence_[1,2]} \kn{} is non-slice and, when we started our work, \g4 was known to be in $[1,2]$ \cite{knotinfo}. We found an unlinking band sequence of \kn{} with length 2. By Remark \ref{remark:interval}, \gds{} $=2$.
    
    \item \label{method:suboptimal_sequence} \kn{} is \{(i) undetermined, (ii) slice but not doubly slice, (iii) non-slice, (iv) non-slice and, when we started our work, \g4 was known to be in $[1,2]$\} \cite{livingston2015doubly,knotinfo}. We found an unlinking band sequence of \kn{} with length $q$ where \{(i) 0, (ii) 1, (iii) \gt24, (iv) 2\} $<q<2g_3(K)$. By Lemma \ref{lemma:genus}, Corollary \ref{corollary:bound}, and Remark \ref{remark:interval}, \{(i) 0, (ii) 1, (iii) \gt24 (iv) 2\} $\leq$ \gds{} $\leq q$. This was the tightest interval determined for \gds{}.
    
    \item \label{method:suboptimal_sequence_Frank} \kn{} is \{(i) undetermined, (ii) slice but not doubly slice\} \cite{livingston2015doubly}. By \cite{owens2021algorithm}, there exists an unlinking band sequence of \kn{} with length $q$ where \{(i) 0, (ii) 1\} $<q<2g_3(K)$. By Corollary \ref{corollary:bound}, \{(i) 0, (ii) 1\} $\leq$ \gds{} $\leq q$. This was the tightest interval determined for \gds{}.
    
    \item \label{method:suboptimal_sequence_unknotting} \kn{} is \{(i) slice but not doubly slice, (ii) non-slice and, when we started our work, \g4 was known to be in $[1,2]$\} \cite{livingston2015doubly,knotinfo}. We have that \{(i) 1, (ii) 2\} $<2u(K)<2g_3(K)$. By Remark \ref{remark:interval}, and Corollary \ref{corollary:unknotting number}, \{(i) 1, (ii) 2\} $\leq$ \gds{} $\leq 2u(K)$. This was the tightest interval determined for \gds{}.
\end{enumerate}

\begin{remark}
For the knots where we have not pinned down \gds{} exactly, we have that \gds{} is contained in either $[0,1]$, $[1,2]$, $[2,3]$ or $[2,4]$. 

Lemma \ref{lemma:p_n} can only be used to find unlinking band sequences of even length. For that reason, it could potentially help our cause further by identifying unlinking band sequences with length 2 for those knots where \gds{} is in $[2,3]$ or $[2,4]$. However, a search for (1,1) unknotting sequences for those knots which had an unknotting number of $\leq2$, or those which had an unknotting number contained in an interval that overlapped with $[1,2]$, yielded none. (Any (1,0) or (0,1) sequences have already been accounted for, either via Corollary \ref{corollary:unknotting number} or the fact that the unlinking band sequence associated with a (1,0) or (0,1) unknotting sequence can be achieved with 2 bands with length 1 and, therefore, would have been identified in our search for unlinking band sequences---see Section \ref{section:automate}.)

In addition, \cite[Theorem~1.1]{orson2020lower} describes a further lower bound on \gds{} using the signature function of \kn{}. However, this lower bound does not improve upon that which we have given.
\end{remark}

\section{Bounds on \gds{}}
\label{section:bounds}

\subsection{$g_3(K)$ and $g_4(K)$}
\label{subsection:genus}
\begin{L3} For a knot \kn{}, $2g_4(K) \stackrel{(1)}{\le} g_{ds}(K) \stackrel{(2)}{\le} 2g_3(K)$.
\end{L3}
\begin{proof}
\hfill \begin{enumerate}
\item Given a closed and orientable surface $W$ smoothly embedded in $S^4 = D^4_1\cup_{S^3}D^4_2$ for which $W\cap S^3 = K$, we have that $W = R_1\cup_KR_2$, where $R_1\subseteq D^4_1$ and $R_2\subseteq D^4_2$ are closed and orientable surfaces with boundary $K$ (also smoothly embedded). Inevitably, the genera of $R_1$ and $R_2$ will be at least $g_4(K)$, which means that the genus of $W$ must be at least $2g_4(K)$.

\item Let $S$ be a Seifert surface of \kn{} which attains the genus \g3. We embed two copies of $S$ into the equatorial $S^3$ of $S^4 = D^4_1\cup_{S^3}D^4_2$. We then push the interior of one copy of $S$ into $D^4_1$, forming $S_+$, and, in a similar fashion, push the interior of the other copy of $S$ into $D^4_2$, forming $S_-$. Let $B = S_+\cup_KS_-$. Note that we can construct $B$ in such a way that $B$ is smoothly embedded in \s4. Since both $S_+$ and $S_-$ have genus \g3, the resulting surface $B$ will have genus $2g_3(K)$. 

We can think of $S$ as being in a disc-band form; let us say that $S$ is composed of $d$ discs and $b$ bands (such that $d-b=1-2g_3(K)$). When we construct $B$, the discs of $S$ become unknotted spheres and the bands become unknotted tubes, so $B$ is made up of $d$ unknotted spheres connected by $b$ unknotted tubes. Consequently, $B$ bounds a handlebody, which means that $B$ is unknotted.

Hence, $B$ is a closed, orientable, and unknotted surface of genus $2g_3(K)$ smoothly embedded in $S^4$ which is divided by $K$. This means that, \gds{} is at most \gt23. \qedhere
\end{enumerate}
\end{proof}

\subsection{Balanced crossing changes} \label{subsection:balance}
\indent \begin{lemma} \label{lemma:crossing change} Given any crossing $C$ of a knot, two oriented saddle moves can be used to engender a crossing change at $C$. \end{lemma}
\begin{proof} Let $u$ and $l$ be the two arcs that make up $C$. We can execute an oriented saddle move from $u$ to itself so that a crossing change occurs at $C$, and we are left with a newly formed loop, $n$, around $l$, as in Figure \ref{fig:crossingchange}. \begin{figure}[h]
    \centering
    \includegraphics{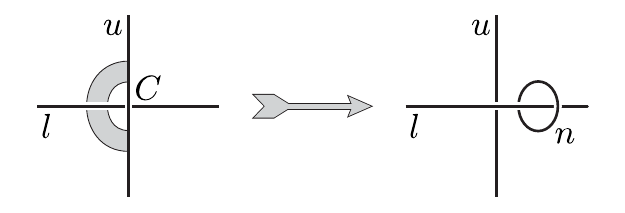}
    \caption{First oriented saddle move required in order to engender a crossing change at C.}
    \label{fig:crossingchange}
\end{figure}
We can then eliminate $n$ with one more oriented saddle move between itself and $l$.
\end{proof}

\begin{L2}
Given a knot \kn{}, there exists an unlinking band sequence with length $2u(K)$. In other words, $g_{ds}(K)\leq 2u(K)$. \end{L2}
\begin{proof} This follows directly from Lemma \ref{lemma:crossing change}. \end{proof}

The upper bound on \gds{} from Corollary \ref{corollary:unknotting number} can be refined by considering the sign of crossing changes.

We say that we have carried out a \ita{balanced crossing change} on a knot \kn{} when we perform two crossing changes on \kn{}, where one is at a positive crossing and the other is at a negative crossing.\begin{lemma} \label{lemma:balanced crossing change} Given a knot \kn{}, let $(P,N)$ be a pair of crossings of \kn{}. If $P$ is positive and $N$ is negative, then there exists two oriented saddle moves which carry out the balanced crossing change of $P$ and $N$.\end{lemma}
\begin{proof} We start by performing the same initial move as in the proof of Lemma \ref{lemma:crossing change} on $P$. This gives rise to a crossing change at $P$ and generates an extra loop around $K$ to spare; let us call this loop $L$. Because $P$ was a positive crossing, $L$ is a positive loop around $K$. If we allow $L$ to travel around $K$ until $L$ reaches either arc of $N$, then we can perform a final oriented saddle move which brings about a crossing change at $N$ and removes $L$. An example of this is shown in Figure \ref{fig:posneg}. \end{proof}
\begin{figure}[H]
    \centering
    \includegraphics{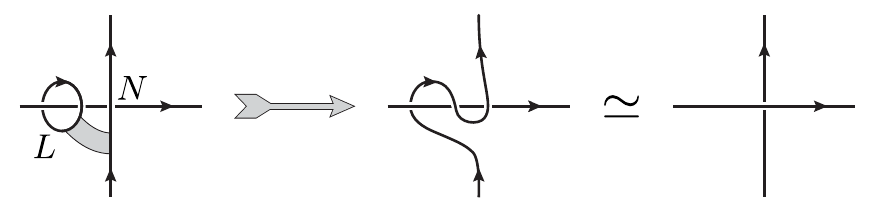}
    \caption{An example of the second oriented saddle move required to accomplish a balanced crossing change.}
    \label{fig:posneg}
\end{figure}

We now recall and give the proof of Lemma \ref{lemma:p_n}.

 \begin{L1} Given a knot \kn{}, if there exists a $(p,n)$ unknotting sequence of \kn{}, then there exists an unlinking band sequence with length $2\cdot \max\{p,n\}$. In other words, \gds{} $\leq 2\cdot \max\{p,n\}$. \end{L1}
\begin{proof} Without loss of generality, assume that $n\leq p$. The $(p,n)$ unknotting sequence can be achieved with $n$ balanced crossing changes and $(p-n)$ normal crossing changes. By Lemma \ref{lemma:crossing change} and Lemma \ref{lemma:balanced crossing change}, we can emulate this unknotting sequence by using $2n+2(p-n)=2p$ oriented saddle moves, i.e., there exists an unlinking band sequence for \kn{} with length $2p$. By Corollary \ref{corollary:bound}, \gds{} $\leq 2p = 2\cdot \max\{p,n\}$.\end{proof}

\section{Determining the smooth 4-genus}
\label{section:4-genus}
\begin{lemma} \label{lemma:interval}
The smooth 4-genus of the following knots is 1: $11n_{80}$, $12a_{187}$, $12a_{230}$, $12a_{317}$, $12a_{450}$ $12a_{570}$, $12a_{624}$, $12a_{636}$, $12a_{905}$, $12a_{1189}$, $12a_{1208}$, $12n_{52}$, $12n_{63}$, $12n_{225}$, $12n_{555}$, $12n_{558}$, $12n_{665}$, and $12n_{886}$.
\end{lemma}
\begin{proof}
For all of the knots listed above, the smooth 4-genus was previously known to be in the interval $[1,2]$ \cite{knotinfo}. Moreover for each of these knots, we found an unlinking band sequence with length $<4$. 

Let $Q$ be any one of the knots from the list above. We shall assume that $g_4(Q)=2$. By Lemma \ref{lemma:genus} and Corollary \ref{corollary:bound}, any unlinking band sequence of $Q$ must have length $\geq4$. But, we have identified an unlinking band sequence with length $<4$. This is a contradiction. Hence, $g_4(Q)=1$.
\end{proof}

\section{Computation}
\label{section:automate}
Computation supporting the results of this paper was performed via a custom module of the Knot-Like Objects (KLO) software \cite{KLO}.  The results of the computation consist of explicit sequences of bands and simplifications showing decomposition of the relevant knots into unlinks. The data files can be downloaded at \path{http://klo-software.net/doublyslice} and browsed via the KLO software.

All searches were performed by enumerating all orientation-preserving bands on a given knot diagram, up to some specified \emph{length} (one plus the number of internal crossings the band makes) and absolute \emph{twist} of the band (relative to the blackboard framing of the diagram); this was performed iteratively after simplifying the resulting knot diagrams, up to some given \emph{depth} (the maximum number of iterations performed). The resulting diagrams were classified as follows:

\begin{itemize}
\item if a diagram could be simplified to one with no crossings, it was declared an \emph{unlink}, directly establising an upper bound for the number of oriented band-moves required to reduce the knot to an unlink (in a small number of cases, this entailed stronger algebraic tangle-simplification methods, in which case the final simplification moves after the final band are not recorded);

\item if a diagram gave a knot having any components with nonzero linking number, gave a hyperbolic knot, or gave a knot whose Alexander invariants were nontrivial, it was declared a \emph{non-unlink};

\item otherwise (and rarely), the diagram was declared a \emph{possible unlink}, as neither of the two other classifications were rigorously established for it as above.
\end{itemize}

\noindent%
Results are recorded in the data files only when the bands exhibited (via an outcome in the first case above) established the tightest upper bound on \gds{}.  A brief summary of the details of the steps taken follows, with $[DLT]$ giving the Depth, maximum Length, and maximum absolute Twist used in each search; each batch of knots starts as a list of knots names and their Dowker-Thistlewaite codes.

The first batch comprised $150$ knots having lower bound $1$; a $[111]$ band search reduced $131$ of these to unlinks, after which $[121]$ reduced an additional $14$, $[131]$ reduced one more, and $[142]$ came up empty.  The four unresolved knots were added to the next batch.

The second batch comprised $1868$ knots having lower bound $2$, along with the four unresolved knots from the previous batch; a $[211]$ band search reduced $1620$ to unlinks, after which $[222]$ reduced an additional $184$.  For each of the remaining $68$ knots, all flype-equivalent diagrams were enumerated, after which $[222]$ reduced an additional $7$ knots to unlinks, leaving $61$ unresolved.  For completeness, these $61$ knots were run through $[311]$ (note that the depth of $3$ in this search doesn't match the proven lower bounds for these knots), which reduced $58$ to unlinks.

The third batch comprised $172$ knots having lower bound $4$; as depth-four searches balloon significantly, an additional stipulation was added: that after simplifying the result of each band, the crossing count must not exceed that of the diagram before the band.  Even enforcing this crossing-monotonicity, a $[411]$ search reduced all $172$ to unlinks.

Six remaining knots having lower bound $6$ were manually reduced to unlinks by the first author using the KLO software, and the sequences of bands used for each were encoded for inclusion with the results from the computations above.

\section{Appendix}
\subsection{Table 1: The doubly slice genera of prime knots with up to 12 crossings} \label{subsection:table}
\hfill\\ In the table below, the superscript of each \gds{} value refers to the list found in Section \ref{section:results}.

\noindent\tiny\centering

\end{tabular}

\vspace{0.5in}

\bibliographystyle{amsalpha}
\bibliography{Bibliography}

\end{document}